%% file: main.tex
\documentclass{article}

\input{style}

\begin{document}

\input{preamble}

\maketitle

\input{abstract}

\section{Introduction}

In this section, we review previous work on magic squares, provide fundamental definitions, and present the manner in which this work is organized.

\subsection{Previous Work}

Magic squares are one of the most well-known topics in recreational mathematics. The notion initially referred to a square array of distinct positive integers in which every row, column, and diagonal sum to the same number. It is well known that magic squares exist for every size larger than a $(2\times2)$ array. In fact, many examples of these were found in antiquity.

This paper originates in our work on the problems posed in the Numberphile video ``The Parker Square---Numberphile'' \cite{parker}. While searching for a magic square of squares, we began considering squares in settings other than the integers. In \cite{cain}  magic squares in finite fields were discussed. The work in this paper began with the goal of defining magic squares in the realm of group theory, at which point we learned that \cite{sun} had initiated an investigation of this idea. The definition we created independently of their work is very similar. However, here, we take the idea further and characterize the finitely generated abelian groups with $(3\times3)$ magic squares.

The previously mentioned 1997 paper by Sun and Yihui \cite{sun} touches on this topic and begins to address the topics on which our work is founded. Lower bounds for the orders of the groups that admit magic squares were established in that paper, but the authors also provide results that determine several classes of groups which have $(n\times n)$ magic squares, leaving open the question of classifications of groups which do not have $(n\times n)$ magic squares. It was proved in the paper that for any positive integer $n \geq 3$, any abelian group of order $n^2$ admits an $(n \times n)$ magic square. They also showed that for any positive prime number $p$ and any positive integer $n$, any elementary abelian $p-$group of size $p^{2n}$ admits a magic square of size $(2n \times 2n)$.

\subsection{Definitions and Notation}

We begin this section with the fundamental definition of this paper.

\begin{definition}[magicgroup]
We say that group $G$ is $n$-\textbf{magic} if it has an $n\times n$ magic square. That is, there exist distinct
$g_{1,1},g_{1,2},\ldots,g_{1,n},g_{2,1},g_{2,2},\ldots,g_{2,n},\ldots,g_{n,1},g_{n,2},\ldots,g_{n,n}\in G$ such that
\begin{align}
g_{1,1}g_{1,2}\ldots g_{1,n}&=g_{2,1}g_{2,2}\ldots g_{2,n}\nonumber \\
&=\ldots\nonumber \\
&=g_{n,1}g_{n,2}\ldots g_{n,n}\nonumber \\
&=g_{1,1}g_{2,1}\ldots g_{n,1}\nonumber \\
&=g_{1,2}g_{2,2}\ldots g_{n,2}\nonumber \\
&=\ldots\nonumber \\
&=g_{1,n}g_{2,n}\ldots g_{n,n}\nonumber \\
&=g_{1,1}g_{2,2}\ldots g_{n,n}\nonumber \\
&=g_{n,1}g_{n-1,2}\ldots g_{1,n}.\nonumber
\end{align}
We call this common product the \textbf{magic product} of the magic square. 
\end{definition}

Note that in \cite{sun} a ``magic group'' refers to the automorphism group on a magic square. This greatly differs from our notion of an $n$-magic group. It is also worth noting that natural variants of ``magic'' objects can be defined similarly, such as for semi-groups or for magmas.  We leave those for a future investigation.  

\subsection{Organization of This Work}

In the remainder of this paper, we first provide some preliminary results on $3$-magic groups. We then prove a series of propositions that act as lemmas for the proof of the characterization of $3$-magic finitely generated abelian groups, which is then presented. Finally, we provide some results related to nonabelian groups and discuss some ideas for future work on this topic.

\section{Results}

The following two results for a group $G$ are immediate. The first provides a lower bound on the order of a group for it to be $n$-magic. The second tells us that if a group is $n$-magic, then any group containing it as a subgroup is also $n$-magic.

\begin{proposition}[minimumorder]
If a group $G$ is $n$-magic, then $|G|\geq n^2$.
\end{proposition}

\begin{proposition}[subgroup]
If $H\leq G$ and $H$ is $n$-magic, then $G$ is $n$-magic.
\end{proposition}

We can quickly see the following.

\begin{theorem}[2-magic]
No groups are 2-magic.
\end{theorem}

\begin{proof}
Suppose $G$ is 2-magic, and let $\begin{bmatrix} g_1 & g_2 \\ g_3 & g_4
\end{bmatrix}$ be a magic square. Then $g_1g_2=g_1g_3$, and through left multiplication by $g_{1}^{-1}$, we have $g_2=g_3$. This contradicts the assumption that the entries in a magic square are distinct.
\end{proof}

\begin{corollary}[cancellation]
    Any monobinary algebra with the cancellation property will not be 2-magic.
\end{corollary}

Given this, we can begin investigating which groups are $3$-magic. From Proposition~\ref{prp:minimumorder}, we have that all groups of order $8$ or less cannot be $3$-magic. We wish to find ``fundamental'' classes of magic groups for which we can then apply Proposition~\ref{prp:subgroup} to find a much larger class of groups. To this end, we endeavour to characterize 3-magic finitely generated abelian groups. The next four propositions are our first steps in this direction.

\begin{proposition}[infinitegroup]
    All infinite, finitely generated abelian groups are $3$-magic.
\end{proposition}

\begin{proof}
    We have that $\Z$ is $3$-magic since
    \begin{gather*}
        \begin{bmatrix}
            8 & 1 & 6  \\
            3 & 5 & 7  \\
            4 & 9 & 2
        \end{bmatrix}
    \end{gather*}
    is a magic square. If $G$ is an infinite finitely generated abelian group then $\Z\leq G$ by the Fundamental Theorem of Finitely Generated Abelian Groups and so $G$ is $3$-magic by Proposition~\ref{prp:subgroup}.
\end{proof}

\begin{proposition}[cyclicgroups]
The cyclic group 
$C_n$ is $3$-magic if and only if $n\geq9$.
\end{proposition}

\begin{proof}
Necessity follows immediately from Proposition~\ref{prp:minimumorder}. For sufficiency, we have that, for $C_n=\lrc{x}$,
\begin{gather*}
    \begin{bmatrix}
        x^{n-3} & x^2       & x         \\
        x^4     & 1         & x^{n-4}   \\
        x^{n-1} & x^{n-2}   & x^3
    \end{bmatrix}
\end{gather*}
forms a magic square with a magic product of $1$. The fact that $n\geq9$ guarantees these entries are distinct.
\end{proof}

\begin{proposition}[cyclicgroups23]
The groups 
$C_{n}^{2}$ and $C_{n}^{3}$ are $3$-magic if and only if $n\geq3$.
\end{proposition}

\begin{proof}
Necessity follows immediately from Proposition~\ref{prp:minimumorder}. For sufficiency, given $C_{n}^{2}$, we must cover a few cases. First, note that $C_n\leq C_{n}^{2}$ for all $n$ so that, by Proposition~\ref{prp:subgroup}, $C_{n}^{2}$ is $3$-magic for all $n\geq9$. Next, for $C_{2k+1}^{2}=\lrc{x,y}$,
\begin{gather*}
    \begin{bmatrix}
        x^k     & y^{k+1}   & x^{k+1}y^k        \\
        xy^k    & 1         & x^{2k}y^{k+1}   \\
        x^ky    & y^k       & x^{k+1}
    \end{bmatrix}
\end{gather*}
forms a magic square with a common product of $1$. Thus, $C_{3}^{2}$, $C_{5}^{2}$, and $C_{7}^{2}$ are $3$-magic. Furthermore, since $C_{3}^{2}\leq C_{6}^{2}$, by Proposition~\ref{prp:subgroup}, $C_{6}^{2}$ is also $3$-magic. Lastly, for $C_{4}^{2}=\lrc{x,y}$,
\begin{gather*}
    \begin{bmatrix}
        x       & y^3   & x^3y      \\
        x^2y    & 1     & x^2y^3    \\
        xy^3    & y     & x^3
    \end{bmatrix}
\end{gather*}
forms a magic square with a common product of $1$. Thus, $C_{4}^{2}$ is $3$-magic. Since $C_{4}^{2}\leq C_{8}^{2}$, by Proposition~\ref{prp:subgroup}, $C_{8}^{2}$ is also $3$-magic.

Finally, since $C_{n}^{2}\leq C_{n}^{3}$ for all $n$, Proposition~\ref{prp:subgroup} provides sufficiency for $C_{n}^{3}$.
\end{proof}

\begin{proposition}[cyclicgroupsnsquaredn]
The group $C_{n}\times C_{n^2}$ is $3$-magic if and only if $n\geq3$.
\end{proposition}

\begin{proof}
Necessity again follows from Proposition~\ref{prp:minimumorder}. Sufficiency follows from Proposition~\ref{prp:subgroup}, after recognizing that $C_{n^2}\leq C_{n}\times C_{n^2}$ .
\end{proof}

To proceed further, we need to define a ``normalized'' magic square.

\begin{definition}[NormalizedMagicSquare]
A \textbf{normalized $n$-magic square} in a group $G$ is an $n$-magic square in $G$ such that one of the entries is the identity element of $G$. When $n$ is odd, if the center entry of the magic square is the identity, we call it a \textbf{centered normalized $n$-magic square}.
\end{definition}

It can be shown that in abelian groups we can convert any magic square to a normalized magic square with the identity in any position that we choose.

\begin{proposition}[normalizing]
    An abelian group is $n$-magic if and only if for all $i,j\in\{1,2,\ldots,n\}$ there exists an $n$-magic square in $G$ whose $g_{i,j}$ entry is the identity. Therefore, an abelian group is $n$-magic if and only if it admits a normalized magic square.
\end{proposition}

\begin{proof}
    Sufficiency is clear. To prove necessity, choose an $n$-magic square in $G$ with magic product $s$. Let $i,j\in\{1,2,\ldots,n\}$ and multiply each entry by $g_{i,j}^{-1}$. Then the resulting square has $1$ in the $i,j$ position, and each row, column, and diagonal has a product of $g_{i,j}^{-n}s$. We can guarantee that all these entries are unique since left multiplication by a group element is an automorphism.
\end{proof}

The next three results answer the remaining question as to whether we can find an even stronger biconditional that can be used to quickly show that an abelian group is not $3$-magic.

\begin{theorem}[abelianparameters]
    The elements of any $3$-magic square in an abelian group can be generated by three elements.
\end{theorem}

\begin{proof}
Sufficiency is immediate. The proof for necessity is adapted from \cite{mathpages} and translated into the language of groups. If $G$ is $3$-magic, let
\begin{gather*}
    \begin{bmatrix}
        g_1     & g_2  & g_3  \\
        g_4     & g_5  & g_6  \\
        g_7     & g_8  & g_9
    \end{bmatrix}
\end{gather*} 
be a magic square in $G$ with magic product $s$. We then have
\begin{align*}
    g_1g_5g_9 &= s \\
    g_2g_5g_8 &= s \\
    g_4g_5g_6 &= s \\
    g_7g_5g_3 &= s \\
    g_1g_4g_7 &= s \\
    g_3g_6g_9 &= s \\
    g_1g_2g_3 &= s \\
    g_7g_8g_9 &= s.
\end{align*}
Multiplying the first four equations gives us
\begin{align*}
    (g_1g_2g_3)(g_4g_5g_6)(g_7g_8g_9)g_{5}^{3} &= s^4 \\
    s^3g_{5}^{3} &= s^4 \\
    g_{5}^{3} &= s
\end{align*}
Then, define $a:=g_{1}g_{5}^{-1}$, $b:=g_{3}g_{5}^{-1}$, and $c:=g_{5}$. Now note that since $g_1g_2g_3=g_{5}^{3}$ we know
\begin{gather*}
    g_2=g_{1}^{-1}g_{3}^{-1}g_{5}^{3},
\end{gather*}
since $g_7g_5g_3=g_{5}^{3}$ we have 
\begin{gather*}
    g_7=g_{3}^{-1}g_{5}^{2},
\end{gather*}
and since $g_1g_5g_9=g_{5}^{3}$ we can state that
\begin{gather*}
    g_9=g_{1}^{-1}g_{5}^{2}.
\end{gather*}
Combining these statements allows us to say that since $g_1g_4g_7=g_{5}^{3}$ we know
\begin{gather*}
    g_4=g_{1}^{-1}g_{5}^{3}g_{7}^{-1}=g_{1}^{-1}g_{5}^{3}(g_{3}^{-1}g_{5}^{2})^{-1}=g_{1}^{-1}g_{3}^{-1}g_{5},
\end{gather*}
since $g_2g_5g_8=g_{5}^{3}$ we have
\begin{gather*}
    g_8=g_{2}^{-1}g_{5}^{2}=(g_{1}^{-1}g_{3}^{-1}g_{5}^{3})^{-1}g_{5}^{2}=g_{1}g_{3}g_{5}^{-1},
\end{gather*}
and since $g_3g_6g_9=g_{5}^{3}$ we have
\begin{gather*}
    g_6=g_{3}^{-1}g_{5}^{3}g_{9}^{-1}=g_{3}^{-1}g_{5}^{3}(g_{1}^{-1}g_{5}^{2})^{-1}=g_{1}g_{3}^{-1}g_{5}.
\end{gather*}
Combining these lets us see that
\begin{align*}
    \begin{bmatrix}
        ac     & a^{-1}b^{-1}c   & bc  \\
        a^{-1}bc     & c    & ab^{-1}c  \\
        b^{-1}c    & abc    & a^{-1}c
    \end{bmatrix} &= 
    \begin{bmatrix}
        g_1g_{5}^{-1}g_{5}    & (g_{1}g_{5}^{-1})^{-1}(g_{3}g_{5}^{-1})^{-1}g_{5}    & g_{3}g_{5}^{-1}g_{5}  \\
        (g_{1}g_{5}^{-1})^{-1}g_{3}g_{5}^{-1}g_{5}     & g_{5}    & g_{1}g_{5}^{-1}(g_{3}g_{5}^{-1})^{-1}g_{5}  \\
        (g_{3}g_{5}^{-1})^{-1}g_{5}    & g_{1}g_{5}^{-1}g_{3}g_{5}^{-1}g_{5}    & (g_{1}g_{5}^{-1})^{-1}g_{5}
    \end{bmatrix} \\
    &= \begin{bmatrix}
        g_{1}    & g_{1}^{-1}g_{3}^{-1}g_{5}^{3}    & g_{3}  \\
        g_{1}^{-1}g_{3}g_{5}     & g_{5}    & g_{1}g_{3}^{-1}g_{5}  \\
        g_{3}^{-1}g_{5}^{2}    & g_{1}g_{3}g_{5}^{-1}    & g_{1}^{-1}g_{5}^{2}
    \end{bmatrix} \\
    &= \begin{bmatrix}
        g_1     & g_2  & g_3  \\
        g_4     & g_5  & g_6  \\
        g_7     & g_8  & g_9
    \end{bmatrix}
\end{align*}
\end{proof}

This leads to the biconditional result that will be used throughout the remainder of the paper.

\begin{corollary}[abelianparameters]
An abelian group $G$ is $3$-magic if and only if there is a centered normalized $3$-magic square in $G$ of the form
\begin{gather*}
    \begin{bmatrix}
        a       & a^{-1}b^{-1}  & b  \\
        a^{-1}b & 1             & ab^{-1}  \\
        b^{-1}  & ab            & a^{-1},
    \end{bmatrix}
\end{gather*}
where $a,b\in G$.
\end{corollary}

\begin{proof}
    This result follows from multiplying every entry in 
    \begin{gather*}
        \begin{bmatrix}
        ac     & a^{-1}b^{-1}c   & bc  \\
        a^{-1}bc     & c    & ab^{-1}c  \\
        b^{-1}c    & abc    & a^{-1}c
    \end{bmatrix}
    \end{gather*}
    by $c^{-1}$.
\end{proof}

Hence, for an abelian group $G$ to be $3$-magic, there must be two elements $a,b\in G$ such that $a$, $a^{-1}b^{-1}$, $b$, $a^{-1}b$, $ab^{-1}$, $b^{-1}$, $ab$, and $a^{-1}$ are distinct. This proves to be especially useful in showing the impossibility of having a $3$-magic square in certain classes of groups.

\begin{corollary}[abelianparameters2]
    The elements of any normalized $3$-magic square in an abelian group can be generated by two elements of the group.
\end{corollary}

\begin{proof}
    From the previous proposition, every $3$-magic square in an abelian group is of the form
    \begin{gather*}
        \begin{bmatrix}
        ac     & a^{-1}b^{-1}c   & bc  \\
        a^{-1}bc     & c    & ab^{-1}c  \\
        b^{-1}c    & abc    & a^{-1}c
        \end{bmatrix}.
    \end{gather*}
    If this were normalized, then one of the entries would be $1$. This would give us an equation in which we could solve for $c$ in terms of $a$ and $b$.
\end{proof}

We can now make use of these results to continue determining which abelian groups are $3$-magic.

\begin{proposition}[cyclicgroups4]
The group $C_{n}^{k}$, with $k\geq4$, is $3$-magic if and only if $n\geq3$.
\end{proposition}

\begin{proof}
Since $C_{n}^{3}\leq C_{n}^{k}$, Proposition~\ref{prp:subgroup} gives sufficiency for $n\geq3$. Now let $n=2$. Assume there is a centered normalized $3$-magic square in $C_{n}^{3}\leq C_{n}^{k}$ given Corollary~\ref{cor:abelianparameters}, and use the notation in that corollary for its form. Since every nontrivial element in $C_{2}^{k}$ is an involution, $a=a^{-1}$, which violates distinctness. Hence, $C_{2}^{k}$ is not $3$-magic for any $k$.
\end{proof}

\begin{proposition}[cyclicgroupsncubedn]
The group $C_{n}\times C_{n^k}$, with $k\geq3$, is $3$-magic if and only if $n\geq2$.
\end{proposition}

\begin{proof}
The fact that $C_{n}\times C_{n^3}$ is $3$-magic for $n\geq3$ follows from Proposition~\ref{prp:subgroup} after recognizing that $C_{n^3}\leq C_{n}\times C_{n^3}$.

For $C_{2^3}\times C_{2}=\lrc{x,y}$, we have that
\begin{gather*}
    \begin{bmatrix}
        xy      & x^5   & x^2y  \\
        x       & 1     & x^7   \\
        x^6y    & x^3   & x^7y
    \end{bmatrix}
\end{gather*}
forms a magic square with a common product of $1$. For $n=2$ and $k\geq4$, we can then use Proposition~\ref{prp:subgroup} in a similar manner as was previously done.
\end{proof}

The following proposition completes the search for abelian $3$-magic groups of odd order and leaves us only with those of even order.

\begin{proposition}[oddorder]
All abelian groups of odd order $n$ with $n\geq9$ are $3$-magic.
\end{proposition}

\begin{proof}
Let $G$ be a finite abelian group of odd order with $|G|=n\geq9$. By the Fundamental Theorem of Finite Abelian Groups, we know that $G=C_{n_1}\times C_{n_2}\times\ldots\times C_{n_{\ell}}$, where $\ell$ is the rank of $G$. If $\ell=1$, then $G$ is $3$-magic by Proposition~\ref{prp:cyclicgroups}. Suppose then that $\ell>1$. Since $\ell$ is the rank of $G$, by the Chinese Remainder Theorem, there must be a $p|n_i$ for $i\in\{1,2,\ldots,\ell\}$ where $p$ is an odd prime. Thus $C_{p}^{\ell}\leq G$ and since $p\geq3$, $G$ is $3$-magic by Propositions~\ref{prp:subgroup}, \ref{prp:cyclicgroups23}, and \ref{prp:cyclicgroups4}.
\end{proof}

We now turn our attention to abelian groups of even order.

\begin{proposition}[2k4]
The group 
$C_{2}^{k}\times C_4$ is not $3$-magic for any $k$.
\end{proposition}

\begin{proof}
Let $C_{2}^{k}\times C_4=<x_1,x_2,\ldots,x_k,y>$, where $|x_i|=2$ (that is, the order of $x_i$ is 2) for all $i$ and $|y|=4$. Assume there is a centered normalized $3$-magic square in $C_{2}^{k}\times C_4$ by Corollary~\ref{cor:abelianparameters}, and use the notation in that corollary for its form. Since $a$ and $b$ cannot be involutions, $|a|=|b|=4$. Then $a=X_1y^{\alpha}$ and $b=X_2y^{\beta}$, where $X_1=x_{i_1}x_{i_2}\ldots x_{i_l}$, $X_2=x_{j_1}x_{j_2}\ldots x_{j_m}$, and $\alpha,\beta\in\{1,3\}$.

\textbf{Case 1 ($\alpha=\beta$):}

Without loss of generality, we may assume that $\alpha=\beta=1$. Then $a^{-1}b=(X_1y^3)(X_2y)=X_1X_2=(X_1y)(X_2y^3)=ab^{-1}$, contradicting distinctness.

\textbf{Case 2 ($\alpha\neq\beta$):} 

Without loss of generality, we may assume that $\alpha=1$ and $\beta=3$. Then $a^{-1}b^{-1}=(X_1y^3)(X_2y)=X_1X_2=(X_1y)(X_2y^3)=ab$, contradicting distinctness.
\end{proof}

In a very similar manner as was previously done, we can determine another family of groups that can be omitted from the search.

\begin{proposition}[2k3]
The group $C_{2}^{k}\times C_3$ is not $3$-magic for any $k$.
\end{proposition}

\begin{proof}
This is true for $k=1$ since $|C_6|<9$. Let $k\geq2$, and let $C_{2}^{k}\times C_3\cong C_{2}^{k-1}\times C_6=<x_1,x_2,\ldots,x_{k-1},y>$, where $|x_i|=2$ for all $i$ and $|y|=6$. Assume there is a centered normalized $3$-magic square in $C_{2}^{k-1}\times C_6$ by Corollary~\ref{cor:abelianparameters}, and use the notation used in that corollary for its form. Since $a$ and $b$ cannot be involutions, $|a|,|b|\in\{3,6\}$. Then $a=X_1y^{\alpha}$ and $b=X_2y^{\beta}$, where  $X_1=x_{i_1}x_{i_2}\ldots x_{i_l}$, $X_2=x_{j_1}x_{j_2}\ldots x_{j_m}$, and $\alpha,\beta\in\{1,2,4,5\}$.

\textbf{Case 1 ($\alpha=\beta$):} 

Then $a^{-1}b=(X_1y^{-\alpha})(X_2y^{\alpha})=X_1X_2=(X_1y^{\alpha})(X_2y^{-\alpha})=ab^{-1}$, contradicting distinctness.

\textbf{Case 2 ($\beta\equiv_6-\alpha$):} 

Then $a^{-1}b^{-1}=(X_1y^{-\alpha})(X_2y^{\beta})=X_1X_2=(X_1y^{\alpha})(X_2y^{-\alpha})=ab$, contradicting distinctness.

\textbf{Case 3 ($\{\alpha,\beta\}=\{1,2\}$):}

Without loss of generality, we may assume $\alpha=1$ and $\beta=2$. Then $a^{-1}b^{-1}=(X_1y^5)(X_2y^4)=X_1y^3X_2=(X_1y)(X_2y^2)=ab$, contradicting distinctness.

\textbf{Case 4 ($\{\alpha,\beta\}=\{1,4\}$):}

Without loss of generality, we may assume $\alpha=1$ and $\beta=4$. Then $a^{-1}b=(X_1y^5)(X_2y^4)=X_1y^3X_2=(X_1y)(X_2y^2)=ab^{-1}$, contradicting distinctness.

\textbf{Case 5 ($\{\alpha,\beta\}=\{2,5\}$):}

Without loss of generality, we may assume $\alpha=2$ and $\beta=5$. Then $a^{-1}b=(X_1y^4)(X_2y^5)=X_1y^3X_2=(X_1y^2)(X_2y)=ab^{-1}$, contradicting distinctness.

\textbf{Case 6 ($\{\alpha,\beta\}=\{4,5\}$):}

Without loss of generality, we may assume $\alpha=4$ and $\beta=5$. Then $a^{-1}b^{-1}=(X_1y^2)(X_2y)=X_1y^3X_2=(X_1y^4)(X_2y^5)=ab$, contradicting distinctness.
\end{proof}

The next proposition may seem less impactful than the previous ones, but its usefulness will be seen in the next section. It turns out knowing whether $C_4\times C_8$ is $3$-magic or not is an exceptional case in the proving of the characterization theorem; the next proposition handles this case.

\begin{proposition}[c4c8]
The group $C_4\times C_8$ is $3$-magic.
\end{proposition}

\begin{proof}
Let $C_4\times C_8=\lrc{x,y}$, where $|x|=4$ and $|y|=8$. Then
\begin{gather*}
    \begin{bmatrix}
        y       & x^3y^6    & xy  \\
        x       & 1         & x^3  \\
        x^3y^7  & xy^2      & y^7
    \end{bmatrix}
\end{gather*}
is a magic square with common product $1$.
\end{proof}

\section{A Characterization Theorem}

We now present the characterization of finitely generated abelian $3$-magic groups, which is built upon many of the previous results of narrower scope.

\begin{theorem}[abeliangroups]
\textbf{(Characterization of the of Finitely Generated Abelian $3$-magic Groups)}

Let $G$ be a finitely generated abelian group. If $G$ is infinite, then it is $3$-magic. If $|G|=n$, then we have the following:
\begin{enumerate}
    \item If $n$ is odd, we know $G$ is $3$-magic if and only if $n\geq9$.
    \item If $n$ is even, we can write its Sylow-2-subgroup in the form $C_{2}^{\alpha_1}\times C_{2^2}^{\alpha_2}\times C_{2^3}^{\alpha_3}\times\ldots\times C_{2^l}^{\alpha_l}$, which lets us say that:
    \begin{enumerate}
        \item If $G$ is a $2$-group, then 
        \begin{enumerate}
            \item if $\alpha_i\neq0$ for some $i\geq4$, we have $G$ is $3$-magic, else
            \item if $\alpha_2\geq2$ or if $\alpha_3\geq2$, we have $G$ is $3$-magic, or else
            \item $G$ is $3$-magic if and only if ($\alpha_1\neq0$ or $\alpha_2=1$) and $\alpha_3=1$.
        \end{enumerate}
        \item If $G$ is not a $2$-group, then 
        \begin{enumerate}
            \item if $p|n$, where $p\geq5$, we have that $G$ is $3$-magic or else
            \item if there is no prime $p\geq5$ such that $p|n$, then
            \begin{enumerate}
                \item if $\alpha_i\neq0$ for some $i\geq2$, we know $G$ is $3$ magic or else
                \item $G$ is $3$-magic if and only if $9|n$.
            \end{enumerate}
        \end{enumerate}
    \end{enumerate}
\end{enumerate}
\end{theorem}

\begin{proof}
Let $G$ be a finitely generated abelian group. If $G$ is infinite then $G$ is $3$-magic by Proposition~\ref{prp:infinitegroup}, so let $G$ be an abelian group with $|G|=n$. Then we have:
\begin{enumerate}
    \item If $n$ is odd, then $G$ is $3$-magic if and only if $n\geq9$ by Proposition~\ref{prp:oddorder}
    \item If $n$ is even, write its Sylow-2-subgroup in the form $C_{2}^{\alpha_1}\times C_{2^2}^{\alpha_2}\times C_{2^3}^{\alpha_3}\times\ldots\times C_{2^l}^{\alpha_l}$.
    \begin{enumerate}
        \item If $G$ is a $2$-group, then
        \begin{enumerate}
            \item if $\alpha_{i_0}\neq0$ for some $i_0\geq4$, then $C_{2^{i_0}}\leq G$, which is $3$-magic by Proposition~\ref{prp:cyclicgroups}. 
            \item if $\alpha_{i}=0$ for all $i\geq4$, then we know that if $\alpha_2\geq2$ or $\alpha_3\geq2$, then $C_{4}^{2}\leq G$ or $C_{8}^{2}\leq G$, both of which are $3$-magic by Proposition~\ref{prp:cyclicgroups23}. Suppose then that $\alpha_{i}=0$ for all $i\geq4$, $\alpha_2\leq1$, and $\alpha_3\leq1$.
            \item Otherwise, we have the following:
            \begin{enumerate}
                \item if $\alpha_1\neq0$ and $\alpha_3=1$, we have $C_2\times C_8\leq G$, which is $3$-magic by Proposition~\ref{prp:cyclicgroupsncubedn}.
                \item if $\alpha_2=1$  and $\alpha_3=1$, then $C_4\times C_8\leq G$, which was shown to be $3$-magic in Proposition~\ref{prp:c4c8}. 
                \item if $\alpha_3=0$, then $G$ is not $3$-magic by Propositions~\ref{prp:cyclicgroups4} and \ref{prp:2k4}.
            \end{enumerate}
        \end{enumerate}
        \item If $G$ is not a $2$-group, then we can say that: 
        \begin{enumerate}
            \item if $p|n$, where $p\geq5$, then $C_{2p}\leq G$, which is $3$-magic by Proposition~\ref{prp:cyclicgroups}.
            \item if there is no prime $p\geq5$ such that $p|n$, then $3|n$ in order for $G$ to not be a $2$-group. It follows that:
            \begin{enumerate}
                \item if $\alpha_{i_0}\neq0$ for some $i_0\geq2$, then $C_{2^{i_0}(3)}\leq G$, which is $3$ magic by Proposition~\ref{prp:cyclicgroups}. 
                \item if $\alpha_i=0$ for all $i\geq2$, then if $9|n$, we have $C_9\leq G$ or $C_{3}^{2}\leq G$, both of which are $3$-magic by Proposition~\ref{prp:cyclicgroups} and \ref{prp:cyclicgroups23}, respectively, and if $9\nmid n$, then $G$ is not $3$-magic by Proposition~\ref{prp:2k3}.
            \end{enumerate}
        \end{enumerate}
    \end{enumerate}
\end{enumerate}
\end{proof}

\section{Nonabelian Groups}

We now briefly turn our attention to nonabelian groups and first present some sufficient conditions for $3$-magic groups.

\begin{proposition}[naprimedivisor]
Let $G$ be a group with $|G|=n$. If any of the following are true, then $G$ is $3$-magic.
\begin{enumerate}
    \item There is a prime $p\geq11$ such that $p|n$.
    \item There is a prime $p\neq2$ such that $p^2|n$.
\end{enumerate}
\end{proposition}

\begin{proof}
\begin{enumerate}
    \item By Cauchy's Theorem, $C_p\leq G$, which is $3$-magic by Propositon~\ref{prp:cyclicgroups}.
    \item As a consequence of the Sylow Theorems, $H\leq G$, where $|H|=p^2$. This means that $H\cong C_{p^2}$ or $H\cong C_{p}^{2}$, both of which are $3$-magic by Proposition~\ref{prp:cyclicgroups23}.
\end{enumerate}
\end{proof}

Next we present a nonabelian group that is $3$-magic but not covered by the previous proposition.

\begin{example}[c7c3]
When the homomorphism is acting faithfully, $C_7\rtimes C_3$ is $3$-magic.
\end{example}

\begin{proof}
When the homomorphism is acting faithfully, we can write $C_7\rtimes C_3=\lrc{a,b\mid a^7=b^3=1, bab^{-1}=a^4}$. Then
\begin{gather*}
    \begin{bmatrix}
        a       & ab        & a^3b^2  \\
        a^2b^2  & 1         & a^6b  \\
        a^2b    & a^5b^2    & a^6
    \end{bmatrix}
\end{gather*}
is a magic square with magic product $1$.
\end{proof}





    



\section{Summary}

    In this paper, we defined a magic group by likening it to a magic square using the language of groups. We were then able to characterize the finitely generated abelian $3$-magic groups. Beyond this, we also discussed whether some nonabelian groups are $3$-magic. 

\section{Future Endeavors}

    There are multiple topics of study for future investigations. For example, we could explore ``magic cubes'' or ``magic tesseracts'' in groups, potentially as an extension of the concept briefly touched on in \cite{sun}. Another idea is to investigate whether various structures (e.g. groups, graphs, etc.) have 3-magic automorphism groups. Finally, in our future study, we would like to find groups that are $n$-magic for $n>3$.

\printbibliography

\end{document}

%% file: style.tex
\usepackage[utf8]{inputenc}
\usepackage[margin=1in]{geometry}
\usepackage{authblk,todonotes,amsmath,amssymb,hyperref,mathrsfs,mathtools,enumitem,setspace,tikz,float,longtable,indentfirst,caption,ragged2e,longtable}
\floatstyle{plaintop}
\restylefloat{table}
\usepackage[utf8]{inputenc}
\usepackage{csquotes}
\usepackage[english]{babel}
\usepackage{pdfpages}

\usepackage{bbm}
\usepackage{dsfont}

\usepackage{multirow}
\usepackage{dcolumn}
\usepackage{todonotes}

\usepackage[backend=biber, style=numeric, sorting=nyt]{biblatex}
\DeclareFieldFormat{pages}{#1}
\renewbibmacro{in:}{%
  \ifentrytype{article}
    {}
    {\bibstring{in}%
     \printunit{\intitlepunct}}}
\usepackage{graphicx}
\usepackage{array}
\graphicspath{ {figures/} }

\usetikzlibrary{shapes.geometric, arrows, positioning}
\tikzstyle{every node}=[font=\small]

\newcommand{\Z}{\mathbb{Z}}

\newcommand{\lrc}[1]{\left<#1\right>}

\addto\captionsenglish{}
\addto\captionsenglish{}
\newcounter{statement}[section]
\renewcommand\thestatement{\thesection.\arabic{statement}}

\newenvironment{definition}[1][]{\refstepcounter{statement}\vspace{12pt}\noindent\textbf{Definition~\thestatement.}\label{def:#1}\rmfamily}{}

\newenvironment{proposition}[1][]{\refstepcounter{statement}\vspace{12pt}\noindent\textbf{Proposition~\thestatement.}\label{prp:#1}\rmfamily\begin{itshape}}{\end{itshape}}

\newenvironment{theorem}[1][]{\refstepcounter{statement}\vspace{12pt}\noindent\textbf{Theorem~\thestatement.}\label{thm:#1}\rmfamily\begin{itshape}}{\end{itshape}}

\newenvironment{corollary}[1][]{\refstepcounter{statement}\vspace{12pt}\noindent\textbf{Corollary~\thestatement.}\label{cor:#1}\rmfamily\begin{itshape}}{\end{itshape}}

\newenvironment{example}[1][]{\refstepcounter{statement}\vspace{12pt}\noindent\textbf{Example~\thestatement.}\label{ex:#1}\rmfamily}{}

\newenvironment{proof}{\vspace{12pt}\noindent\textbf{Proof.}}{\hfill$\square$}
\allowdisplaybreaks

\makeatletter
\let\@fnsymbol\@arabic
\makeatother

%% file: preamble.tex
\title{Introducing $n$-Magic Groups and Characterizing $3$-Magic Finitely Generated Abelian Groups}

\author{Danielle Bowerman\thanks{dcbhz9@umsystem.edu --- supported by the Chancellor's Distinguished Fellows Program}, Nicholas Fleece\thanks{nfg4h@umsystem.edu --- supported by the Kummer Institute for Student Success, Research and Economic Development},  Matt Insall\thanks{insall@mst.edu}}
\affil{Missouri University of Science and Technology}

%% file: abstract.tex
\begin{abstract}
    In this paper, we define an $n$-magic square in a group to be an $(n\times n)$ array of group elements whose rows, columns, and diagonals have the same product. This definition is akin to the idea of magic squares in the integers. Groups that have an $n$-magic square are said to be $n$-magic. We begin with some preliminary results and focus much of our attention on $3$-magic groups. Through a series of propositions, we ultimately prove a characterization theorem for $3$-magic finitely generated abelian groups. We then discuss some additional results about non-abelian groups as well as $n$-magic groups where $n>3$.
\end{abstract}